\documentclass[12pt]{amsart}

\usepackage{amssymb,latexsym,amsthm,amsmath,amscd,ytableau}
\usepackage{graphicx,epstopdf}
\usepackage{fullpage}
\usepackage{url,hyperref,cleveref}

\newtheorem{theorem}{Theorem}
\newtheorem{lemma}{Lemma}
\newtheorem{prop}{Proposition}
\newtheorem{corollary}{Corollary}

\numberwithin{subcase}{case}

\theoremstyle{definition}

\theoremstyle{remark}

\newtheorem*{remark*}{Remark}

\newcommand\C{\mathbb{C}}
\newcommand\CP{\mathbb{CP}}

\newcommand\Z{\mathbb{Z}}

\newcommand\IMG{\mathrm{IMG}}
\newcommand\inv{^{-1}}
\renewcommand\Re{\mathop{\mathrm{Re}}}
\newcommand\id{\operatorname{id}}

\newcommand{\mx}[1]{\mathbf{#1}}
\newcommand{\bdy}{\partial}
\newcommand{\conj}[1]{\overline{#1}}
\newcommand{\EP}{\mathbb{E}^2}
\newcommand{\EPclos}{\overline{\mathbb{E}}{}^2}
\newcommand{\UD}{\mathbb{D}}

\begin{document}
\title{Geometry and Algebra of the Deltoid Map}
\markright{The Deltoid Map}
\author{Joshua P. Bowman}
\date{\today}

\maketitle

\begin{abstract}
The geometry of the deltoid curve gives rise to a 
self-map of $\C^2$ that is expressed in coordinates by 
$f(x,y) = (y^2 - 2x, x^2 - 2y)$. This is one in a family 
of maps that generalize Chebyshev polynomials to several 
variables. We use this example to illustrate two important 
objects in complex dynamics: the Julia set and the iterated 
monodromy group.
\end{abstract}

\section{Introduction.}

Complex dynamics is perhaps best known for the fractal 
images it produces. For instance, given a polynomial 
function $\C \to \C$, an important set to consider is 
the \emph{Julia set}, whose points behave ``chaotically'' 
under iteration of the function; for most polynomials, 
the Julia set is a fractal. However, the Julia set is 
a smooth curve in the case of two special families: 
\emph{power maps}, having the form $z \mapsto z^d$, 
and \emph{Chebyshev polynomials}, of which the simplest 
example is $z \mapsto z^2 - 2$. For power maps, the 
Julia set is the unit circle, and for Chebyshev polynomials 
it is the segment $[-2,2]$, contained in the real line. 
These structurally simple examples play a distinguished 
role in complex dynamics, and studying them can illuminate 
parts of the theory that apply in more complicated cases.

Power maps have an obvious generalization to functions 
from $\C^n$ to itself: just take the $d$th power of each 
coordinate. The higher-dimensional analogues of Chebyshev 
polynomials are not as obvious, however. In the 1980s, 
Veselov \cite{aV86,aV91} and Hoffman--Withers \cite{mHwW88} 
independently constructed a family of ``Chebyshev-like'' 
self-maps of $\C^n$ associated to each crystallographic 
root system of rank $n$. The cases where $n = 2$ have 
received much further attention (see, e.g., 
\cite{aL90,sN08,bRhM11,kU01,kU07,kU09,wW88}), 
especially for the $A_2$ root system, which is connected 
with the deltoid curve (a.k.a.\ three-cusped hypocycloid 
or Steiner's hypocycloid).

This article presents a new approach to construct a 
quadratic $A_2$-type map $f$ based directly on 
geometric properties of the deltoid. For this reason 
we call $f$ \emph{the deltoid map}. The set of lines 
tangent to the deltoid will play a crucial role, and 
indeed we will see that $f$ preserves this set of lines. 
(This fact was previously observed in \cite{wW88}; the 
difference is that we construct the map from the tangent 
lines, rather than starting with the map ahead of time 
and deducing from it the invariance of the tangent lines; 
in particular, our approach does not use the theory of root 
systems.) Using this invariance property, we will study 
two dynamical features of $f$: one geometric (the Julia set) 
and the other algebraic (the iterated monodromy group). 
Both of these objects will be formally defined later in the 
article.

%We find a geometric description of $J$ 
%(Theorem~\ref{T:pedal}) that leads to an algebraic 
%equation for $J$, and we compute $\mathrm{IMG}(f)$ 
%explicitly (Theorem~\ref{T:img}).

The Julia set of $f$ is a real algebraic hypersurface $J$ 
of degree $4$ (Corollary~\ref{C:julia}). We derive this 
property from a description of $J$ in terms of pedal 
curves, which arise from classical differential geometry 
(Theorem~\ref{T:pedal}). The Julia set of $f$ is therefore 
considerably more interesting geometrically than in 
the case of a Chebyshev polynomial in one variable, 
the segment $[-2,2]$ mentioned above.

The iterated monodromy group of $f$ is an affine Coxeter 
group (Theorem~\ref{T:img}). Such groups are present 
implicitly in the construction from \cite{aV86,aV91} and 
explicitly in \cite{mHwW88}. The connection with iterated 
monodromy groups is new, however, and extends the 
(very short) list of polynomial endomorphisms of $\C^n$, 
with $n \ge 2$, whose iterated monodromy groups are known 
(see \cite{jBsK10,vN12} for the only other examples known 
to the author).

In future work, we will show how these properties of the 
deltoid map generalize to other Chebyshev-like maps.

\section{Lines and planes.}

In this section we establish some notation and terminology.

The \emph{complex projective line} $\CP^1$ is identified 
with the one-point compactification of $\C$ 
(i.e., the Riemann sphere) in the usual way; generally 
$t \in \C \cup \{\infty\}$ will be used to mean 
this extended complex coordinate on $\CP^1$. 
The \emph{complex projective plane} $\CP^2$ has 
homogeneous coordinates $[x:y:z]$, where $x$, $y$, and $z$ 
are complex numbers, not all zero; this means that  
$[x:y:z] = [\alpha x:\alpha y:\alpha z]$ for all 
$\alpha \in\C\setminus\{0\}$. We use $[a:b:c]^\vee$ 
to represent homogeneous coordinates on the 
\emph{dual projective plane} $(\CP^2)^\vee$, whose 
elements are the lines in $\CP^2$, so that 
\[
[x:y:z] \; \in \; [a:b:c]^\vee 
\qquad\iff\qquad
ax + by + cz = 0.
\]
The \emph{affine plane} $\C^2$ is canonically included in 
$\CP^2$ via the map $(x,y) \mapsto [x:y:1]$. The complement 
of the image of $\C^2$ under this embedding is the 
{\em complex line at infinity} $L_\infty \cong \CP^1$, 
having equation $z = 0$; that is, $L_\infty = [0:0:1]^\vee$.

The real plane in $\C^2$ with equation $y = \bar{x}$ is 
a copy of the Euclidean plane, 
and it will be denoted by $\EP$. Its closure in $\CP^2$ 
is a copy of the real projective plane, 
$\EPclos \cong \mathbb{RP}^2$, but we do not write it 
as such, because the coordinates induced on $\EP$ as 
a subset of $\C^2$ are not real. We call $\bdy\EP = 
\EPclos \setminus \EP = \EPclos \cap L_\infty \cong S^1$ the 
\emph{circle at infinity}, trusting no confusion will arise 
from the fact that the (real) circle at infinity is contained 
in the (complex) line at infinity.

As a real submanifold of $\CP^2$, $\EP$ does not carry a 
complex structure (else its closure could not be the real 
projective plane, topologically), but the restriction of 
the coordinate $x$ to $\EP$ provides a bijection 
$\EP \cong \C$. This is what we will always mean when we 
carry out constructions on $\EP$ using a complex coordinate.

\section{The deltoid as a real curve and as a complex curve.}
\label{S:curve}

In this section we collect some known 
properties of the deltoid---especially regarding 
its tangent lines---that will be useful in our study.

\begin{figure}[h]\centering
\includegraphics{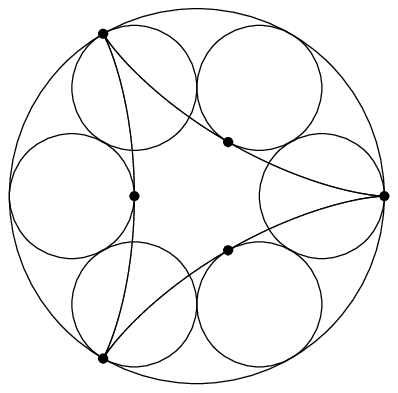}\hspace{0.9in}
\includegraphics{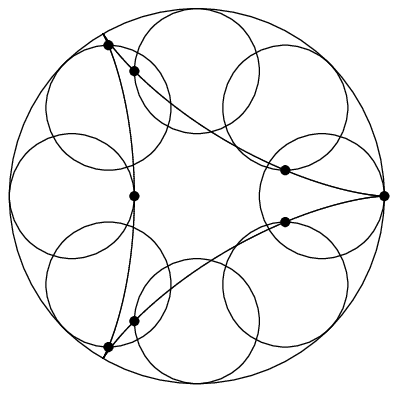}
\caption{Tracing out the deltoid as a hypocycloid.}
\label{F:deltoidtrace}
\end{figure}

The classical \emph{deltoid} 
is the curve traced by a point marked on the 
circumference of a circle of radius $1$ rolling 
without slipping inside a circle of radius $3$. 
When the center of the smaller circle travels once 
counterclockwise around the center of the larger circle, 
a point on the smaller circle's circumference makes two 
clockwise revolutions around its center. 
(See Figure~\ref{F:deltoidtrace}.) 
Because the centers remain $2$ units apart, the deltoid 
can be parametrized in $\EP$ by 
\[
x = 2t + {\bar{t}}^{2}, \qquad |t| = 1.
\]
This extends to a complex algebraic curve in the following 
way. Because $\EP$ is embedded in $\C^2$ as the real plane 
$y = \bar{x}$, the parametrization of the deltoid in $\C^2$ 
becomes $(2t+{\bar{t}}^2,2\bar{t}+t^2)$ with $|t|=1$. In 
order to make this parametrization holomorphic, we replace 
$\bar{t}$ with $t^{-1}$ (when $|t| = 1$, these are the same), 
and we define 
\begin{equation}\label{E:delparam}
\gamma(t) = \left(2t + \frac{1}{t^2}, \frac{2}{t} + t^2\right), 
\qquad t \in \C\setminus\{0\}.
\end{equation}
We can further extend $\gamma$ to a curve in $\CP^2$, 
which we also call $\gamma$, by appending an additional 
coordinate, initially equal to $1$, then clearing 
denominators (which is allowed in homogeneous coordinates):
\[
\gamma(t) = \big[ 2t^3 + 1 : 2t + t^4 : t^2 \big], 
\qquad t \in \CP^1
\]
Note that $\gamma(0) = [1:0:0]$ and $\gamma(\infty) = [0:1:0]$. 
(To see why the latter expression is correct, rewrite $\gamma(t)$ 
as $\gamma(1/s)$, clear denominators, then let $s$ go to $0$.) 
These are the only two points of $\CP^1$ that $\gamma$ sends to 
$L_\infty$.

$\mathcal{D}$ will denote the image of $\gamma$ in either 
$\C^2$ or $\CP^2$, and $\mathcal{D}_{\EP} = \mathcal{D} \cap \EP$ 
is the \emph{real} deltoid.

In $\C^2$ we have 
\[
\gamma'(t) 
= \left(2 - \frac{2}{t^3}, -\frac{2}{t^2} + 2t\right) 
= 2 \left(1 - \frac{1}{t^3}\right) (1,t)
\]
and so $\gamma'(t)$ vanishes precisely when $t$ equals $1$, 
$\omega = e^{i\,2\pi/3}$, or $\omega^2 = e^{i\,4\pi/3}$; 
these cube roots of unity give rise to the three cusps 
of $\mathcal{D}$. At every other point of $\mathcal{D}$, a 
tangent vector is $(1,t)$. An equation for the line tangent 
to $\mathcal{D}$ at $\gamma(t)$ is therefore 
\[
\begin{vmatrix}
1 & x - 2t - t^{-2} \\
t & y - 2t^{-1} - t^2 
\end{vmatrix} = 0,
\]
which is equivalent to 
\begin{equation}\label{E:taneqn}
t^3 - t^2 x + t y - 1 = 0.
\end{equation}
This equation works equally well at the cusps, where 
$t^3 = 1$ and \eqref{E:taneqn} reduces to $y = tx$, 
so each cusp also has a well-defined tangent line, 
which passes through the origin. 

It is worth remarking that in \cite{fMfvM54} 
the study of the real deltoid begins, not with 
any classical construction, but with equation 
\eqref{E:taneqn}, restricted to $y = \bar{x}$ 
and $|t| = 1$, which is simply called the 
``line equation'' of the deltoid.

Equation \eqref{E:taneqn} shows that a generic point 
$(x,y)$ of $\C^2$ lies on three tangent lines of 
$\mathcal{D}$. A point belongs to $\mathcal{D}$ 
if and only if at least two of these tangent lines 
coincide, which is to say that the discriminant of 
the left side of \eqref{E:taneqn} (as a polynomial in $t$) 
is zero. Thus we obtain an affine equation for 
$\mathcal{D}$ (and an additional reason to name this set 
$\mathcal{D}$, since it is where a discriminant vanishes): 
\begin{equation}\label{E:affeqn}
%\[
%\begin{vmatrix}
%1 & -x & y & -1 & 0 \\
%0 & 1 & -x & y & -1 \\
%3 & -2x & y & 0 & 0 \\
%0 & 3 & -2x & y & 0 \\
%0 & 0 & 3 & -2x & y 
%\end{vmatrix} 
x^2 y^2 - 4\big(x^3 + y^3\big) + 18 xy - 27 = 0.
%\]
\end{equation}

Now we can also parametrize the dual curve 
$\mathcal{D}^\vee$ in the dual projective plane 
$(\CP^2)^\vee$. From \eqref{E:taneqn}, we get the 
following parametrization of $\mathcal{D}^\vee$:
\begin{equation}\label{E:dualparam}
\check\gamma(t) = [-t^2:t:t^3-1]^\vee.
\end{equation}
In particular we see that 
$\check\gamma(0) = \check\gamma(\infty) = [0:0:1]^\vee$, 
so that the line at infinity in $\CP^2$ is tangent to 
$\mathcal{D}$ at both $\gamma(0)$ and $\gamma(\infty)$. 
(This tangency can also be seen, less directly, from the 
fact that $L_\infty$ intersects $\mathcal{D}$, a curve of 
degree 4, in only two points.)
From \eqref{E:dualparam}, we can deduce that an 
equation for $\mathcal{D}^\vee$ is 
\begin{equation}\label{E:dualeqn}
a^3 + b^3 = abc
\end{equation}
(when $a$, $b$, and $c$ are real, this equation produces 
the folium of Descartes). This curve is smooth except 
for a self-intersection at $[0:0:1]^\vee$, which shows 
that the line at infinity is the only bitangent of 
$\mathcal{D}$.

Because equation \eqref{E:affeqn} has degree four, a 
generic line in $\CP^2$ will intersect $\mathcal{D}$ in 
four points. Meanwhile, a generic element of $\mathcal{D}^\vee$ 
(that is, a line tangent to $\mathcal{D}$) will intersect 
$\mathcal{D}$ at two points besides the point of tangency. 
These other two points of intersection are connected with 
several interesting geometric properties; we state three 
of them here for later use. All three have easy algebraic 
proofs, which we leave to the reader. They are illustrated 
in Figure~\ref{F:3properties}.

\begin{figure}[h]\centering
\includegraphics{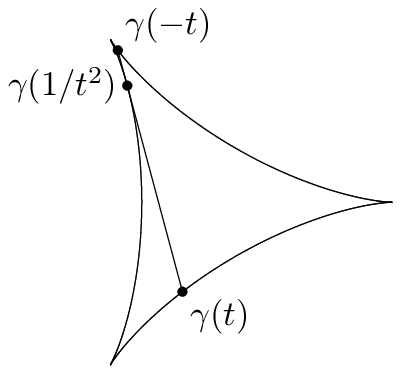}
\hspace{0.35in}
\includegraphics{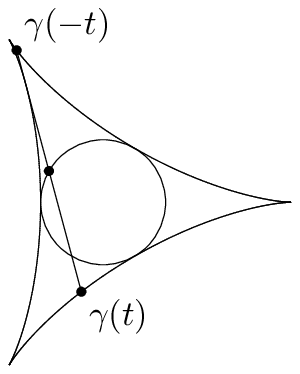}
\hspace{0.35in}
\includegraphics{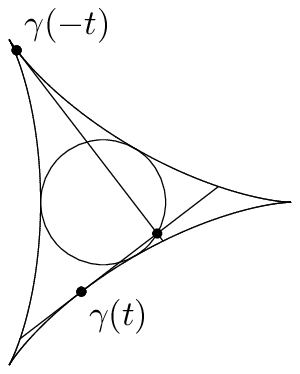}

\caption{Three properties of lines tangent to $\mathcal{D}$.}
\label{F:3properties}
\end{figure}

\begin{enumerate}
\item[(A)] For all $t \in \C\setminus\{0\}$, the line containing 
$\gamma(t)$ and $\gamma(-t)$ is tangent to $\mathcal{D}$ at 
$\gamma(1/t^2)$. 
\item[(B)] The midpoint of $\gamma(t)$ and $\gamma(-t)$ 
in $\C^2$ lies on the curve $xy = 1$. 
\item[(C)] The tangent lines $\check\gamma(t)$ and 
$\check\gamma(-t)$ intersect at a point also on $xy = 1$.
\end{enumerate}

Property (A) will later form the basis for our 
geometrically-defined dynamical system. Properties 
(B) and (C) will relate to the critical points of the map.

The curve $\mathcal{C}$ with equation $xy = 1$ is, 
projectively speaking, a conic section. Its 
intersection $\mathcal{C}_{\EP}$ with the plane $\EP$ 
is the unit circle, having equation $|x|^2 = 1$. 
%The figures below illustrate in $\EP$ the three 
%properties just proved.

The real deltoid $\mathcal{D}_{\EP}$ is a Jordan curve 
in $\EP$; let $K$ be the union of $\mathcal{D}_{\EP}$ 
with its interior. $K$ consists of those 
points $x$ such that all solutions to 
$t^3 - xt^2 + \bar{x}t - 1 = 0$ lie on 
the unit circle $|t| = 1$; in other words, these are 
the points that lie on three ``real'' tangent lines. 
(See Figure~\ref{F:K}, left and middle.) 
%In $\EP \setminus K$, each point lies on one real tangent 
%and two complex lines that are tangent to $\mathcal{D}$ 
%but do not touch $\mathcal{D}_{\EP}$.

\begin{figure}[h]\centering
\includegraphics{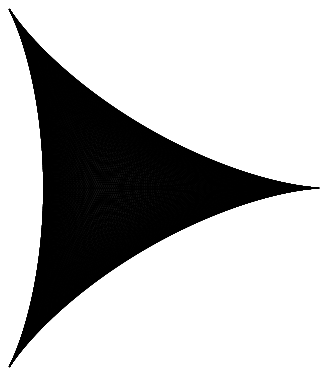}\hspace{0.45in}
\includegraphics{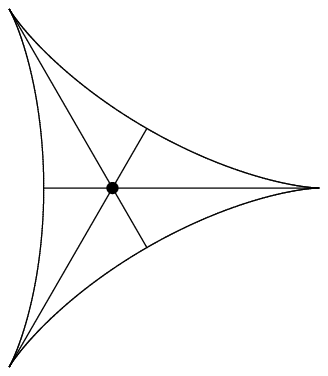}\hspace{0.45in}
\includegraphics{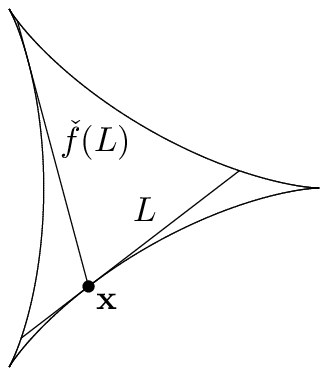}
\caption{{\sc Left:} The set $K \subset \EP$ bounded by 
$\mathcal{D}\cap{\EP}$. {\sc Middle:} Tangent lines through 
the three cusps of $\mathcal{D}$ and their point of 
intersection at the origin.
{\sc Right:} For a generic tangent line $L \in \mathcal{D}^\vee$ 
there is another line $\check{f}(L) \in \mathcal{D}^\vee$ such 
that $\check{f}(L)$ is secant to $\mathcal{D}$ at the point 
$\mathbf{x}$ where $L$ is tangent.}\label{F:K}
\end{figure}

\section{The deltoid map.}\label{S:map}

In this section we use the geometric properties of 
$\mathcal{D}$ to define a map $f$ from $\CP^2$ to itself. 
First, we define a natural map 
$\check{f}$ on the dual curve $\mathcal{D}^\vee$. 
Given $L = T_\mx{x}\mathcal{D} \in \mathcal{D}^\vee$, 
let $\check{f}(L)$ be the unique element of $\mathcal{D}^\vee$ 
such that $\{L,\check{f}(L)\}$ is the full set of tangent lines 
to $\mathcal{D}$ passing through $\mx{x}$, as illustrated 
in Figure~\ref{F:K}, right. 
(Note that $\check{f}(L)$ is the same as $L$ when $\mx{x} = \gamma(t)$ for 
$t \in \{1,\omega,\omega^2,0,\infty\}$, and it is distinct 
otherwise.) 
%Then set $\check{f}(L) = L'$. 
It follows from property (A) in the previous section that 
\[
\check{f}(\check\gamma(t)) = \check\gamma(1/t^2) 
\qquad\text{for all $t \in \CP^1$.}
\]
In particular, $\check{f}$ fixes $L_\infty$ as an element 
of $\mathcal{D}^\vee$, but it is helpful to think of it as 
exchanging the points of tangency, namely $\gamma(0)$ and 
$\gamma(\infty)$.

Now we turn to our promised self-map of $\C^2$. First 
we observe that, given $(x,y) \in \C^2$, the solutions 
$t_1, t_2, t_3$ to \eqref{E:taneqn} satisfy $t_1 t_2 t_3 = 1$ 
and 
\begin{equation}\label{E:xyparam}
x = t_1 + t_2 + t_3, \qquad
y = \frac{1}{t_1} + \frac{1}{t_2} + \frac{1}{t_3}.
\end{equation}
Conversely, if $t_1, t_2, t_3$ are chosen to satisfy 
$t_1 t_2 t_3 = 1$, then the formulas \eqref{E:xyparam} 
provide coefficients for the equation \eqref{E:taneqn} 
to be solved by $t_1, t_2, t_3$.

\begin{prop}
Suppose $\check\gamma(t_1)$, $\check\gamma(t_2)$, and 
$\check\gamma(t_3)$ are concurrent. Then so are 
$\check{f}(\check\gamma(t_1))$, $\check{f}(\check\gamma(t_2))$, 
and $\check{f}(\check\gamma(t_3))$.
\end{prop}
\begin{proof}
If the point of concurrency lies on $L_\infty$, 
then the result is trivial, as at most two lines 
are involved. Otherwise a necessary and sufficient 
condition for concurrency is $t_1 t_2 t_3 = 1$. 
But if $t_1$, $t_2$, and $t_3$ satisfy this equality, 
then also $(1/{t_1}^2) (1/{t_2}^2) (1/{t_3}^2) = 
(1/t_1 t_2 t_3)^2 = 1$.
\end{proof}

This proposition provides the basis for defining a map on 
all of $\CP^2$: given $\mx{x} \in \CP^2$, let $L_1$, $L_2$, 
and $L_3$ be the three elements of $\mathcal{D}^\vee$ 
passing through $\mx{x}$ (some of these may coincide). Then 
define $f(\mx{x})$ to be the point at which $\check{f}(L_1)$, 
$\check{f}(L_2)$, and $\check{f}(L_3)$ are concurrent. 
(See Figure~\ref{F:geomdef}.) To handle the special cases 
of when all three lines $L_1$, $L_2$, and $L_3$ coincide, 
we extend by continuity and define $f([1:0:0]) = [0:1:0]$, 
$f([0:1:0]) = [1:0:0])$, and whenever 
$\check{f}(L_1) = \check{f}(L_2) = \check{f}(L_3)$ passes 
through a cusp of $\mathcal{D}$, $f(\mx{x})$ is defined 
to be that cusp.

\begin{figure}\centering
\includegraphics{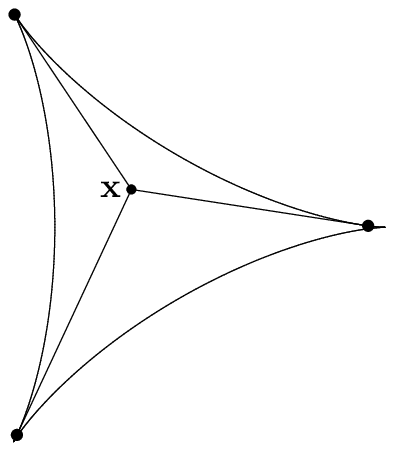}
\hspace{0.1in}
\includegraphics{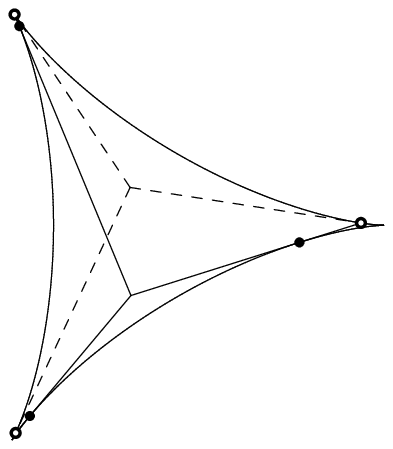}
\hspace{0.1in}
\includegraphics{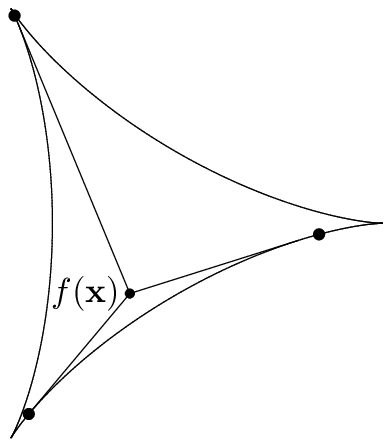}

\caption{Geometric definition of $f$. Any point 
$\mx{x} \in \CP^2$ lies on three tangent lines 
of $\mathcal{D}$ (counted with multiplicity). 
The point of tangency for each of these lines 
lies on another element of $\mathcal{D}^\vee$, 
as seen in Figure~\ref{F:K}. The resulting collection 
of three new tangent lines (again, counted with 
multiplicity) is concurrent at $f(\mx{x})$.}\label{F:geomdef}
\end{figure}

With this geometric definition in hand, we find polynomials 
that describe $f$.

\begin{prop}\label{P:formulas}
On $\C^2$, $f$ takes the form 
$(x,y) \mapsto (y^2 - 2x, x^2 - 2y)$. 
On $\CP^2$, this extends to 
$[x:y:z] \mapsto [y^2 - 2xz : x^2 - 2yz : z^2]$. 
On $L_\infty$, $f$ has the form $\zeta \mapsto 1/\zeta^2$.
\end{prop}

\begin{proof}
If $(x,y) \in \C^2$, and $t_1$, $t_2$, and $t_3$ are the 
roots of \eqref{E:taneqn}, then by the observations 
surrounding equation~\eqref{E:xyparam}, we have 
\[
f(x,y) 
= \left( 
  \frac{1}{{t_1}^2} + \frac{1}{{t_2}^2} + \frac{1}{{t_3}^2},\;
  {t_1}^2 + {t_2}^2 + {t_3}^2
  \right).
\]
Now we observe that 
\[
\left( \frac{1}{t_1} + \frac{1}{t_2} + \frac{1}{t_3} \right)^2 
- 2 \left( 
    \frac{1}{t_1 t_2} + \frac{1}{t_2 t_3} + \frac{1}{t_3 t_1} 
    \right)
= \frac{1}{{t_1}^2} + \frac{1}{{t_2}^2} + \frac{1}{{t_3}^2}
\]
and 
\[
(t_1 + t_2 + t_3)^2 - 2\, (t_1 t_2 + t_2 t_3 + t_3 t_1) 
= {t_1}^2 + {t_2}^2 + {t_3}^2,
\]
which proves the result on $\C^2$. The formula on 
$\CP^2$ is then obtained by a standard homogenization 
process. Because $L_\infty$ is defined by $z = 0$, 
on this line the map becomes $[x:y:0] \mapsto [y^2:x^2:0]$; 
if we set $\zeta = y/x$, the result for $L_\infty$ 
becomes clear. Alternatively, for $L_\infty$ 
we could use the observations made previously that 
$f(\check\gamma(t)) = \check\gamma(1/t^2)$ and that 
$\check\gamma(t)$ intersects $L_\infty$ at $[1:t:0]$, 
so $\zeta = t$.
\end{proof}

\section{Julia set, Fatou set, and Green function.}

Having defined the deltoid map $f$, we turn to some of 
its dynamical properties. Ideally, for any point 
$\mx{x} \in \CP^2$, we would like to be able to predict 
the behavior of its \emph{orbit} under $f$, which is the 
sequence $\mx{x},f(\mx{x}),f^2(\mx{x}),f^3(\mx{x}),\dots$, 
and also to say something about the orbits of points near 
$\mx{x}$. (Here and in the rest of the article $f^n$ denotes 
the composition of $f$ with itself $n$ times; this notation 
is standard in dynamical systems.) From the construction of 
$f$, we can already see that it has some exceptional 
properties: the deltoid $\mathcal{D}$ is forward invariant, 
meaning $f(\mathcal{D}) = \mathcal{D}$, and $f$ also sends 
each line tangent to $\mathcal{D}$ to another such line. 
These tangent lines will continue to be key in studying 
properties of $f$. 

Notice that $f$ commutes with the involution $\iota(x,y) 
= (y,x)$. The composition $\iota \circ f = f \circ \iota$ 
is studied by Uchimura in \cite{kU01,kU07,kU09} and Nakane 
in \cite{sN08}. The dynamical properties of $f$ and 
$\iota \circ f$ are essentially identical.

A fundamental tool in complex dynamics is the partition 
of the dynamical space into the Fatou set, where the 
dynamics are ``simple,'' and the Julia set, where the 
dynamics are ``chaotic.'' More precisely, the Fatou set 
$\Omega = \Omega_f$ is the largest open set of $\CP^2$ on 
which the iterates of $f$ locally form an equicontinuous 
family; thus if $\mx{x}$ and $\mx{y}$ are points of $\Omega$ 
that are sufficiently near each other, then $f^n(\mx{x})$ 
and $f^n(\mx{y})$ remain close (in $\CP^2$) as $n$ increases. 
The Julia set $J = J_f$ is the complement of $\Omega$; thus 
if $\mx{x}$ is in $J$ and $\mx{y}$ is close to $\mx{x}$, then 
$f^n(\mx{x})$ and $f^n(\mx{y})$ may be very far apart.

On $L_\infty$, as we have seen, $f$ reduces to the power 
map $\zeta \mapsto 1/\zeta^2$. This map of $\CP^1$ exchanges 
$0$ and $\infty$ (in $\CP^2$, these are the points $[1:0:0]$ 
and $[0:1:0]$), and so these two points form a period $2$ 
orbit. If $|\zeta| \ne 1$, then $\zeta^{(-2)^n}$ approaches 
the previously observed period $2$ orbit. If $|\zeta| = 1$, 
then $\zeta^{(-2)^n}$ remains on the unit circle, while 
some nearby points are drawn to the $\{0,\infty\}$ orbit. 
Thus the Julia set of $f$ on $L_\infty$ is the circle at 
infinity, and the Fatou set in $L_\infty$ has two components, 
one containing $0$ and the other $\infty$.

To determine the Julia and Fatou sets of $f$ in $\C^2$, 
we introduce the \emph{Green function} $G = G_f$ of $f$, 
which is defined \cite{eBmJ00,jHpP94} by 
\[
G(\mx{x}) 
= \lim_{n\to\infty} 
  \frac{1}{2^n} \log^+ \left\|f^n(\mx{x})\right\|,
\]
where ${\log^+} = \max\,\{\log,0\}$, and $\|\cdot\|$ 
is any norm on $\C^2$. This function measures how 
quickly points of $\C^2$ escape to infinity under 
iteration of $f$; it is zero precisely for those points 
whose orbits are bounded, which comprise the set $K$. 
It is a continuous, subharmonic function on $\C^2$, and 
it satisfies the functional equation $G(f(x,y)) = 2G(x,y)$.

For most self-maps of $\C^2$, the Green function cannot 
be explicitly calculated. The deltoid map is an exception.

\begin{prop}\label{P:green}
The Green function $G$ of the deltoid map $f$ can 
be calculated as follows: given $(x,y) \in \C^2$, let $t_1$, 
$t_2$, and $t_3$ be the solutions to \eqref{E:taneqn}. 
Then 
\begin{equation}\label{E:green}
G(x,y) 
= \log \max 
  \left\{ |t_1|,\, |t_2|,\, |t_3|,\, 
          \frac{1}{|t_1|},\, \frac{1}{|t_2|},\, \frac{1}{|t_3|} 
  \right\}.
\end{equation}
\end{prop}

Notice that we do not need to use $\log^+$ in 
\eqref{E:green}, because the set over which the 
maximum is taken contains at least one element 
that is greater than or equal to $1$.

\begin{proof}[Proof of Proposition~\ref{P:green}]
Using the $L^\infty$ norm on $\C^2$, we have 
\[
G(x,y) = \lim_{n\to\infty} 
\frac{1}{2^n} {\log^+} \max \left\{ 
  \left| {t_1}^{2^n} + {t_2}^{2^n} + {t_3}^{2^n} \right|,\; 
  \left| \frac{1}{{t_1}^{2^n}} + \frac{1}{{t_2}^{2^n}} 
        + \frac{1}{{t_3}^{2^n}} \right| \right\}.
\]
Set $\tau = \max \big\{ |t_1|, |t_2|, |t_3|, |t_1|\inv, 
|t_2|\inv, |t_3|\inv \big\}$. Then $\tau \ge 1$, and we have 
\begin{gather}
\frac{1}{2^n} \log \max 
\left\{ 
  \left| {t_1}^{2^n} + {t_2}^{2^n} + {t_3}^{2^n} \right|,\, 
  \left| \frac{1}{{t_1}^{2^n}} + \frac{1}{{t_2}^{2^n}} 
        + \frac{1}{{t_3}^{2^n}} \right| \right\}
- \log \tau \label{Eq:greendiff} \\
= \frac{1}{2^n} \log \max 
\left\{ 
  \frac{\left| {t_1}^{2^n} + {t_2}^{2^n} + {t_3}^{2^n} \right|}
       {{\tau}^{2^n}},\, 
  \frac{1}{{\tau}^{2^n}}
  \left| \frac{1}{{t_1}^{2^n}} + \frac{1}{{t_2}^{2^n}} 
        + \frac{1}{{t_3}^{2^n}} \right| \right\}. \label{Eq:greenquot}
\end{gather}
By our choice of $\tau$, the maximum of the set in 
\eqref{Eq:greenquot} is bounded by $3$. Therefore, as $n$ tends to 
$\infty$, the difference in \eqref{Eq:greendiff} tends to $0$. This 
shows that $G(x,y) = \log\tau$, as claimed.
\end{proof}

In terms of the Green function, $\Omega$ is the 
set of points where $dd^c\,G$ vanishes. Here 
$dd^c = \frac{i}{2\pi}\partial\conj\partial$ 
is the so-called \emph{pluri-Laplacian}, and the 
derivatives should properly be interpreted as 
currents (``differential forms with distributional 
coefficients''); for us, however, it is sufficient 
to know where $dd^c\,G = 0$. Because $dd^c\,{\log^+}|t|$ 
vanishes except on the unit circle $S^1$, we obtain 
the following characterization of $J$.

\begin{prop}\label{P:julia}
The Julia set of $f$ is the set $J$ of points 
$[x:y:z] \in \CP^2$ such that the polynomial 
$z(t^3 - 1) - xt^2 + yt$ has at least one root on $S^1$.
\end{prop}

Given our geometric definition of $f$, this result 
is not surprising: as we have seen, the line 
$\check\gamma(t) \in \mathcal{D}^\vee$ intersects 
$L_\infty$ at $[1:t:0]$, and the circle at infinity, 
where $|t| = 1$, is precisely the Julia set of 
$f|_{L_\infty}$.%, which is to say $\zeta \mapsto 1/\zeta^2$.

Nakane \cite{sN08} provided a description of the foliation 
of $J$ by ``stable disks'' of the circle at infinity, 
as well as how external rays land at points of $K$. 
We shall take a different perspective and consider the 
intersection of $J$ with complex lines in $\C^2$ parallel 
to the $x$- and $y$-axes. In order to describe the result, 
however, we must invoke some classical differential geometry. 

Given a curve $C$ and a point $O$ in $\EP$, the 
{\em pedal curve of $C$ with respect to $O$} is the 
locus of points $P$ such that $P$ is the orthogonal 
projection of $O$ onto a line tangent to $C$. (See 
Figure~\ref{F:pedal} for some examples.)

\begin{figure}[h]

\begin{tabular}{rll}
\includegraphics{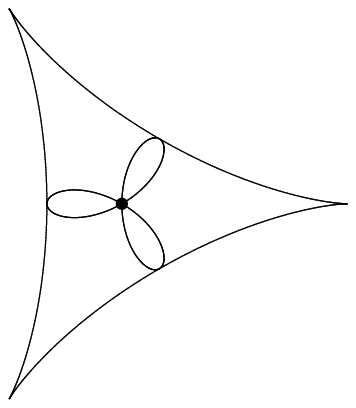}&
\includegraphics{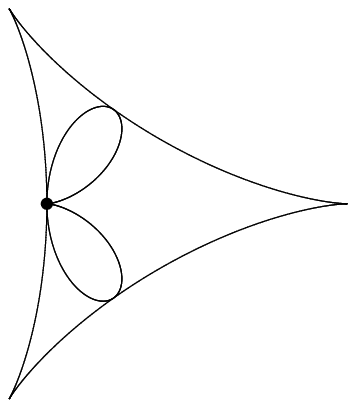}&
\includegraphics{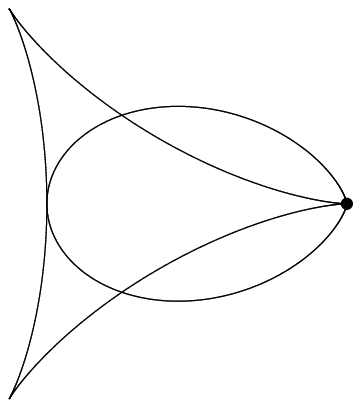}\\
\includegraphics{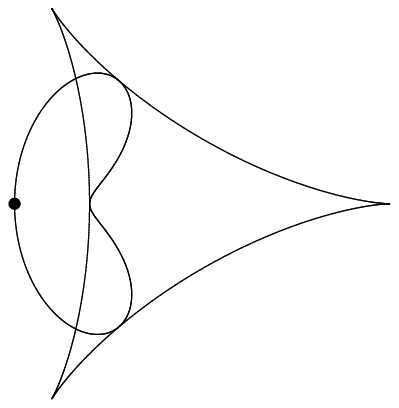}&
\includegraphics{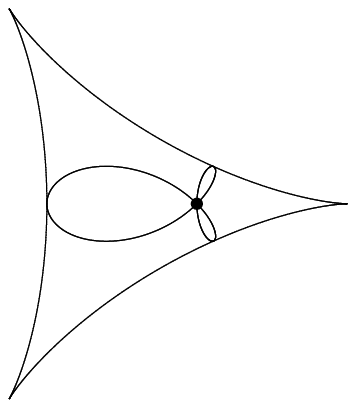}&
\includegraphics{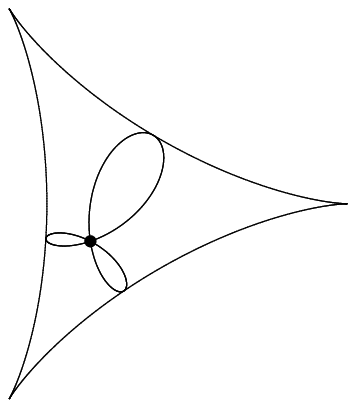}
\end{tabular}

\caption{Some pedal curves of the real deltoid in $\EP$. 
In each image the point $O$ is indicated by a dot. 
{\sc Top:} With respect to the center (a trifolium), 
with respect to the point opposite a cusp (a bifolium), 
and with respect to a cusp (a simple folium).
{\sc Bottom:} With respect to an exterior point, 
with respect to an interior point on an axis of symmetry, 
and with respect to a generic interior point.}\label{F:pedal}
\end{figure}

At this point we can state our first main result, 
which says that the Julia set of the deltoid map on 
$\C^2$ geometrically decomposes into a disjoint union 
of pedal curves of the real deltoid.

\begin{theorem}\label{T:pedal}
The intersection of $J$ with a line $L \ne L_\infty$ 
through $[1:0:0]$ (that is, parallel to the $x$-axis 
in $\C^2$) is the pedal curve of the real deltoid with 
respect to the $x$-coordinate of $L \cap \EP$. 
Likewise, the intersection of $J$ with a line parallel 
to the $y$-axis is the pedal curve of the real deltoid 
with respect to the $y$-coordinate of the intersection 
of this line and $\EP$.
\end{theorem}

To prove this result, we will use the following 
projection from $\C^2$ to $\EP$:
\[
\mathrm{pr}_{\EP}(x,y) 
= \left(\frac{x+\bar{y}}{2},\frac{y+\bar{x}}{2}\right).
\]
This projection is orthogonal with respect to the 
standard Hermitian inner product on $\C^2$, namely 
$(x_1,y_1)\cdot(x_2,y_2) = x_1 \overline{x_2} + y_1 \overline{y_2}$. 
Conveniently, it also preserves each complex line that is 
tangent to $\mathcal{D}$ at a point of $\mathcal{D}_{\EP}$, 
which is the content of the next lemma.

\begin{lemma}
If $|t| = 1$ and $(x,y) \in \check\gamma(t)$, then also 
$\mathrm{pr}_{\EP}(x,y) \in \check\gamma(t)$.
\end{lemma}
\begin{proof}
By assumption, $t$ and $(x,y)$ satisfy equation 
\eqref{E:taneqn} $t^3 - t^2 x + ty - 1 = 0$, 
as well as its conjugate 
$\bar{t}^3 - \bar{t}^2 \bar{x} + \bar{t} \bar{y} - 1 = 0$. 
Because $|t| = 1$, we have $\bar{t} = t^{-1}$, and so, 
after multiplying the conjugate of \eqref{E:taneqn} by 
$t^3$ we obtain $1 - t \bar{x} + t^2 \bar{y} - t^3 = 0$. 
Subtracting this latter equation from \eqref{E:taneqn} 
and dividing by $2$ produces 
\[
t^3 - t^2 \!\left(\frac{x + \bar{y}}{2}\right) 
+ t \!\left(\frac{y + \bar{x}}{2}\right) - 1 = 0
\]
as desired.
\end{proof}

A line in $\C^2$ parallel to the $x$-axis is determined 
by its $y$-coordinate. Let $L_\alpha$ be the line with 
equation $y = \bar\alpha$. The intersection of $L_\alpha$ 
with $\EP$ is 
\[
L_\alpha \cap \EP = \{(\alpha,\bar\alpha)\}.
\]
The restriction of $\mathrm{pr}_{\EP}$ to $L_\alpha$ is 
a bijection, whose inverse $\lambda_\alpha : \EP \to L_\alpha$ 
is the affine map 
\[
\lambda_\alpha(x,\bar{x}) = (2x - \alpha, \bar\alpha).
\]
Notice, however, that with respect to the metrics 
induced on $\EP$ and $L_\alpha$ by the Hermitian inner 
product on $\C^2$, $\lambda_\alpha$ is not just affine, but 
a similarity. To prove Theorem~\ref{T:pedal}, therefore, 
it suffices to show that $\mathrm{pr}_{\EP}(J \cap L_\alpha)$ 
is the pedal curve of $\mathcal{D} \cap \EP$ with respect to 
$(\alpha,\bar\alpha)$. Or what is the same, we need to show 
that for all $t \in S^1$, the point $(x,\bar{x}) \in \EP$ 
is the orthogonal projection of $(\alpha,\bar\alpha)$ 
onto $\check\gamma(t) \cap \EP$ if and only if 
$\lambda_\alpha(x,\bar{x})$ is in $\check\gamma(t)$.

If $|t| = 1$, then the Hermitian inner product of the vectors 
$(1,t)$ and $(1,-t)$ is zero, so any two lines in $\C^2$ 
of the form $y = tx + b_1$ and $y = -tx + b_2$ are orthogonal.

\begin{proof}[Proof of Theorem~\ref{T:pedal}]
Let $|t| = 1$. The intersection of 
$\check\gamma(t)$ and $\EP$ has the equation 
\[
t^3 - t^2 x + t \bar{x} - 1 = 0
\text{,\qquad or}\qquad
\bar{x} = tx - t^2 + t\inv\text.
\]
The line through $(\alpha,\bar\alpha)$ that is 
orthogonal to $\check\gamma(t)$ is therefore 
\[
\bar{x} - \bar\alpha = -t(x - \alpha)\text.
\]
These latter two equations together imply 
(by eliminating $\bar{x}$) that 
\[
tx - t^2 + t\inv = \bar\alpha - t(x - \alpha)\text,
\]
and solving for $x$ produces 
\[
x = \frac{1}{2} \big( \alpha + t + \bar\alpha t\inv - t^{-2}\big)\text.
\]
On the other hand, if $\lambda_\alpha(x,\bar{x}) \in J$, 
then 
\[
t^3 - t^2 (2x - \alpha) + t\bar\alpha - 1 = 0\text,
\]
which produces the same solution for $x$, as desired.

The proof for the intersection of $J$ with a line 
parallel to the $y$-axis is virtually identical.
\end{proof}

From this geometric description of the intersection of 
$J$ with a horizontal or vertical line, we can find an 
algebraic equation for $J$ in $\C^2$. 

\begin{corollary}\label{C:julia}
The Julia set of $f$ is the real hypersurface in $\C^2$
having the equation
\[
2 \Re (x - \bar{y})^3 + \Re (x - \bar{y})^2 (\bar{x}^2 - y^2) = 0.
\]
\end{corollary}
 
\begin{proof}
Start in $\EP$ with the real lines 
\[
t^3 - t^2 x + t \bar{x} - 1 = 0
\qquad\text{and}\qquad
\bar{x} - \bar\alpha = -t(x - \alpha)\text,
\]
then eliminate $t$ to get
\[
\left(\frac{\bar{x} - \bar\alpha}{x - \alpha}\right)^3 
+ \left(\frac{\bar{x} - \bar\alpha}{x - \alpha}\right)^2 x 
+ \left(\frac{\bar{x} - \bar\alpha}{x - \alpha}\right) \bar{x} 
+ 1 = 0
\]
%or 
%\[
%(\bar{x} - \bar\alpha)^3 
%+ (\bar{x} - \bar\alpha)^2 (x - \alpha) x 
%+ (\bar{x} - \bar\alpha)(x - \alpha)^2\bar{x} 
%+ (x - \alpha)^3 = 0
%\]
Now a point $(x,y) \in \C^2$ is in $J$ if 
$\mathrm{pr}_{\EP}(x,y)$ satisfies this equation 
(meaning we replace $x$ with $(x+\bar{y})/2$ 
and $\bar{x}$ with $(\bar{x}+y)/2$) 
when $\alpha = \bar{y}$, which yields
%\[
%\left(\frac{y+\bar{x}}{2} - \bar\alpha\right)^3 
%+ \left(\frac{y+\bar{x}}{2} - \bar\alpha\right)^2 
%  \left(\frac{x+\bar{y}}{2} - \alpha\right) 
%  \left(\frac{x+\bar{y}}{2}\right) 
%+ \left(\frac{y+\bar{x}}{2} - \bar\alpha\right)
%  \left(\frac{x+\bar{y}}{2} - \alpha\right)^2
%  \left(\frac{y+\bar{x}}{2}\right)
%+ \left(\frac{x+\bar{y}}{2} - \alpha\right)^3 = 0
%\]
%\[
%\left(\frac{y+\bar{x}}{2} - y\right)^3 
%+ \left(\frac{y+\bar{x}}{2} - y\right)^2 
%  \left(\frac{x+\bar{y}}{2} - \bar{y}\right) 
%  \left(\frac{x+\bar{y}}{2}\right) 
%+ \left(\frac{y+\bar{x}}{2} - y\right)
%  \left(\frac{x+\bar{y}}{2} - \bar{y}\right)^2
%  \left(\frac{y+\bar{x}}{2}\right)
%+ \left(\frac{x+\bar{y}}{2} - \bar{y}\right)^3 = 0
%\]
%\[
%\left(\frac{-y+\bar{x}}{2}\right)^3 
%+ \left(\frac{-y+\bar{x}}{2}\right)^2 
%  \left(\frac{x-\bar{y}}{2}\right) 
%  \left(\frac{x+\bar{y}}{2}\right) 
%+ \left(\frac{-y+\bar{x}}{2}\right)
%  \left(\frac{x-\bar{y}}{2}\right)^2
%  \left(\frac{y+\bar{x}}{2}\right)
%+ \left(\frac{x-\bar{y}}{2}\right)^3 = 0
%\]
%\[
%2(\bar{x} - y)^3 
%+ (\bar{x} - y)^2 
%  (x - \bar{y}) 
%  (x + \bar{y}) 
%+ (\bar{x} - y)
%  (x - \bar{y})^2
%  (y + \bar{x})
%+ 2(x - \bar{y})^3 = 0
%\]
\[
2(\bar{x} - y)^3 
+ (\bar{x} - y)^2 
  (x^2 - \bar{y}^2) 
+ (x - \bar{y})^2
  (\bar{x}^2 - y^2)
+ 2(x - \bar{y})^3 = 0\text.
\]
This is equivalent to the desired equation.
\end{proof}

Note that in particular the equation in 
Corollary~\ref{C:julia} is satisfied when 
$y = \bar{x}$, so $\EP$ is entirely contained in $J$. 
This is to be expected, because every point of $\EP$ 
lies on a line that intersects $L_\infty$ on the 
circle at infinity.

To end this section, we provide a description of 
the Fatou set $\Omega$.

\begin{corollary}\label{T:fatou}
$\Omega$ has two components, each of which is 
biholomorphic to $(\UD \times \UD)/\sigma$, where 
$\UD$ is the open unit disk in $\C$ and $\sigma$ 
is the involution $\sigma(u,v) = (v,u)$. These two 
components are exchanged by $f$.
\end{corollary}

\begin{proof}[Proof of Corollary~\ref{T:fatou}]
Define the following two functions from $\C^2$ to $\CP^2$:
\begin{align*}
\Psi_x(u,v) 
&= \big[ u^2 v + uv^2 + 1 : u + v + u^2 v^2 : uv \big]\text, \\
\Psi_y(u,v) 
&= \big[ u + v + u^2 v^2 : u^2 v + uv^2 + 1 : uv \big]\text.
\end{align*}
Direct computation shows that  
\[
(f \circ \Psi_x)(u,v) = \Psi_y(u^2,v^2) 
\qquad\text{and}\qquad
(f \circ \Psi_y)(u,v) = \Psi_x(u^2,v^2)\text,
\]
and for $uv \ne 0$, $\Psi_x(1/u,1/v) = \Psi_y(u,v)$. 
Geometrically, $u$ and $v$ are the $t$-parameters for 
two of the lines in $\mathcal{D}^\vee$ passing through 
$\Psi_x(u,v)$, the third being $1/uv$. Thus, $\Psi_x(u,v)$ 
is contained in $J$ if and only if either $u$ or $v$ lies 
on the unit circle, and the same holds for $\Psi_y(u,v)$. 
Together, $\Psi_x$ and $\Psi_y$ cover all of $\CP^2$.

By definition of $J$ as the complement of $\Omega$, 
we see that $\Omega$ is covered by the two images of 
$\UD \times \UD$ via $\Psi_x$ and $\Psi_y$. Thus $\Omega$ 
has two connected components. The polynomials defining 
$\Psi_x$ and $\Psi_y$ are symmetric in $u$ and $v$, and 
distinct unordered pairs $\{u,v\}$ lead to different 
points of $\CP^2$ by $\Psi_x$ and $\Psi_y$. 
This proves the result.
\end{proof}

The functions $\Psi_x$ and $\Psi_y$ are variants of 
the function $\Psi$ used in \cite{sN08} as an 
``inverse B\"ottcher coordinate'' on the Julia set of 
$f$. We can see from the formulas for $f \circ \Psi_x$ 
and $f \circ \Psi_y$ how the orbit of any point of 
$\Omega$ tends uniformly and super-exponentially to 
the orbit consisting of 
$\Psi_x(0,0) = [1:0:0]$ and $\Psi_y(0,0) = [0:1:0]$.

\section{Iterated monodromy group of the deltoid map.}

We begin this final section with one more exceptional 
property of $f$.

The Jacobian determinant of $f$ at $(x,y) \in \C^2$ is 
$4(1 - xy)$. Thus the locus of critical points in $\C^2$ 
is the curve $\mathcal{C}$ having equation $xy = 1$, whose 
importance was previously noted in Section \ref{S:curve}. Indeed, 
because the lines $\check\gamma(t)$ and $\check\gamma(-t)$ 
have the same image under $\check{f}$, their point of 
intersection must be a critical point of $f$; by property 
(C), all such points lie on $\mathcal{C}$.

If we parametrize $\mathcal{C}$ by $(t,1/t)$, then we find 
that the image of a point of $\mathcal{C}$ can be written as 
\[
f\!\left(t,\frac1t\right) 
= \left(-2t + \frac{1}{t^2}, t^2 - \frac{2}{t}\right) 
= \gamma(-t),
\]
and so we see that $f(\mathcal{C}) = \mathcal{D}$. Because 
$\mathcal{D}$ is forward invariant under $f$, we conclude 
that $f$ is \emph{post-critically finite}, meaning that the 
post-critical locus $\bigcup_{n\ge1} f^n(\mathcal{C})$ is an 
algebraic curve---in this case, $\mathcal{D}$ itself. 
(Post-critically finite maps of $\CP^2$ were introduced in 
\cite{jeFnS92}, under the name of ``critically finite rational 
maps.'')

Set $\mathcal{X} = \C^2 \setminus \mathcal{D}$ and 
$\mathcal{X}_1 = \mathcal{X} \setminus \mathcal{C}$. 
Then the above property implies that 
%$f\inv(\mathcal{X}) = \mathcal{X} \setminus \mathcal{C}$, 
$f \vert_{\mathcal{X}_1}$ is a covering map 
from $\mathcal{X}_1$ to $\mathcal{X}$, called a 
\emph{partial self-covering} of $\mathcal{X}$. 
Let $\mx{x}_0 = (0,0) \in \mathcal{X}$; then the 
fundamental group $\pi_1(\mathcal{X},\mx{x}_0)$ permutes 
the set of preimages of $\mx{x}_0$ by $f$ in a standard 
way: given $[\eta] \in \pi_1(\mathcal{X},\mx{x}_0)$ 
and $\mx{y} \in f^{-1}(\mx{x}_0)$, use $f$ to lift 
$\eta$ to a path $\bar\eta$ starting at $\mx{y}$, and 
let $[\eta] \cdot \mx{y}$ be the endpoint of $\bar\eta$. 
This defines a homomorphism $\mu_f$ from 
$\pi_1(\mathcal{X},\mx{x}_0)$ to the symmetric group 
on $f^{-1}(\mx{x}_0)$, called the \emph{monodromy 
homomorphism}.

Likewise, if we set $\mathcal{X}_n = f^{-n}(\mathcal{X})$, 
then $f^n \vert_{\mathcal{X}_n}$ is a covering map, and 
$\pi_1(\mathcal{X},\mx{x}_0)$ acts on $f^{-n}(\mx{x}_0)$ 
by the monodromy homomorphism $\mu_{f^n}$. The intersection 
\[
\kappa_f = \bigcap_{n\ge1} \mathop{\mathrm{ker}}\mu_{f^n}
\]
is a normal subgroup of $\pi_1(\mathcal{X},\mx{x}_0)$, 
consisting of all elements $[\eta]$ such that every lift 
of $\eta$ by every iterate of $f$ remains a loop. The 
quotient 
\[
\mathrm{IMG}(f) = \pi_1(\mathcal{X},\mx{x}_0)/\kappa_f
\]
is called the {\em iterated monodromy group} of $f$. 
(See \cite{sG12,vN11} for details.)

Iterated monodromy groups are a relatively recent 
addition to the complex dynamics toolbox. They have 
already proved useful in classification problems 
\cite{lBvN06} and in determining the shape of Julia sets 
more complicated than that of the deltoid map \cite{vN12}. 
Nevertheless, only a few such groups have been 
explicitly calculated, especially for maps in dimension 
greater than $1$. A nice feature of $f$ is that 
$\mathrm{IMG}(f)$ can be found directly from the 
definition, which is how we will prove our second 
main result.

\begin{theorem}\label{T:img}
$\mathrm{IMG}(f)$ is isomorphic to the affine 
Coxeter group $\tilde{A}_2$.
\end{theorem}

$\tilde{A}_2$ can be realized geometrically as the group 
generated by reflections across the sides of an equilateral 
triangle in the plane. It has the group presentation 
\[
\tilde{A}_2 = 
\left\langle g_1,g_2,g_3 \mid 
\forall k\ g_k^2 = 1,\; 
\forall j \forall k\ (g_jg_k)^3 = 1 \right\rangle\text.
\]
On the other hand, the fundamental group 
$\pi_1(\mathcal{X},\mx{x}_0)$ is isomorphic to 
the related Artin group 
\[
\bar{A}_2 = \langle h_1,h_2,h_3 \mid 
\forall j \forall k\ h_jh_kh_j = h_kh_jh_k \rangle
\]
(see \cite{eaBjicA08} for a proof). Note that in 
$\tilde{A}_2$, the relation $(g_jg_k)^3 = 1$ is 
equivalent to $g_j g_k g_j = g_k g_j g_k$, and so 
$\tilde{A}_2$ can be obtained from $\bar{A}_2$ by 
adding the relations $h_k^2 = 1$ for $k = 1,2,3$. 
We will accomplish this in Lemma~\ref{L:order2}, 
then show that no additional relations are present 
in $\mathrm{IMG}(f)$.

First we find a useful set of generators for 
$\pi_1(\mathcal{X},\mx{x}_0)$: these can be chosen as 
circles contained in the lines $\check\gamma(\omega)$, 
$\check\gamma(\omega^2)$, and $\check\gamma(1)$ and 
passing through $\mx{x}_0$. To see why, we use the 
Zariski--van~Kampen theorem \cite{eV33,oZ29}, which 
states that generators can be obtained by taking a 
sufficiently general line $L$ and drawing loops around 
the finite set of points $L \cap \mathcal{D}$. 
The condition on $L$ is that $L \cap \mathcal{D}$ 
should have four distinct points in $\C^2$. 
We choose a line of the form 
$L = \{ (x,y) \mid x + y = -a \}$, where $2 < a < 3$. 
Then \eqref{E:affeqn} implies that 
$\gamma(t)$ lies on $L$ if 
\[
t^4 + 2t^3 + at^2 + 2t + 1 = 0,
\]
and our choice of $a$ ensures that all solutions 
of this equation lie on the unit circle, which means 
all points of intersection in $L \cap \mathcal{D}$ lie 
in $\EP$. (See Figure~\ref{F:pi1}, left.) Thus the 
four points of $L \cap \mathcal{D}$ lie in a straight 
(real) line, and so we can draw small loops around 
these inside the (complex) line $L$. Each such loop 
intersects $\EP$ in two points: one in $K$, and one outside.

\begin{figure}[h]\centering
\includegraphics{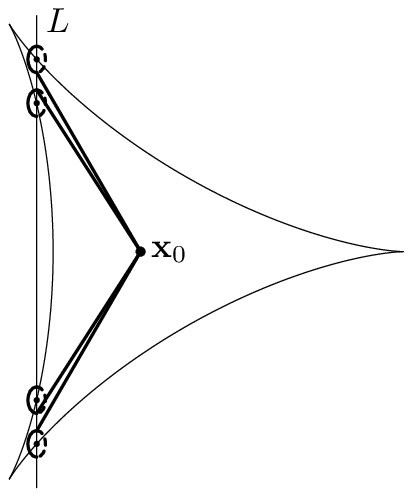}\hspace{.65in}
\includegraphics{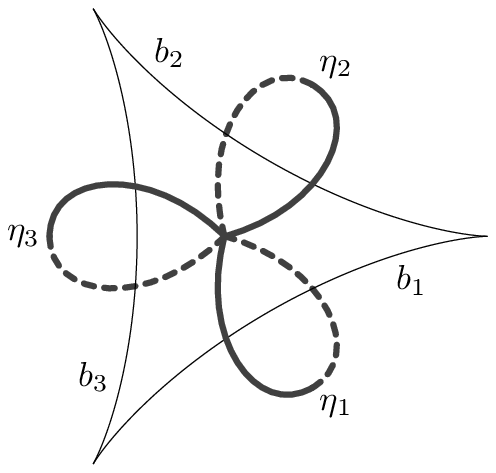}
\caption{{\sc Left:} The complex line $L$ with equation 
$x+y=a$ intersects the deltoid $\mathcal{D}$ at four 
points, all contained in $\EP$, provided $-3 < a < -2$. 
Around each point of intersection, draw a loop inside 
$L$ that intersects $K$ at one point. When connected 
to $\mx{x}_0$ by additional segments in $K$, these loops 
generate $\pi_1(\mathcal{X},\mx{x}_0)$. 
{\sc Right:} Generators for $\pi_1(\mathcal{X},\mx{x}_0)$, 
homotopic to those found in left picture. 
Each loop $\eta_k$ is contained in the complex line 
$\check\gamma(\omega^k)$, which intersects $\mathcal{D}$ 
at the cusp $\gamma(\omega^k)$ and at the midpoint of 
the opposite branch $b_k$.}
\label{F:pi1}
\end{figure}

Connect each loop in $L$ from the point where it 
intersects $K$ to $\mx{x}_0$ with a line segment, 
so that it becomes an element of $\pi_1(\mathcal{X},\mx{x}_0)$ 
(with orientation given by the complex line in which it lies). 
Let's label these elements. The real deltoid has 
three cusps, and between these lie three ``branches'': 
\begin{itemize}
\item one from $\gamma(1)$ to $\gamma(\omega)$, 
\item one from $\gamma(\omega)$ to $\gamma(\omega^2)$, and 
\item one from $\gamma(\omega^2)$ to $\gamma(1)$.
\end{itemize}
Call these branches, respectively, $b_2$, $b_3$, and 
$b_1$, so that $b_k$ and $b_{k+1}$ meet at the cusp to 
which $\check\gamma(\omega^{k+2})$ lies tangent. (All 
indices are computed modulo $3$.) The loops in $L$ 
surrounding $b_1$ and $b_2$ are homotopic in $\mathcal{X}$ 
to loops that lie in $\check\gamma(\omega)$ and 
$\check\gamma(\omega^2)$. On the other hand, the two loops 
surrounding $b_3$ are both homotopic to the {\em same} loop 
in $\check\gamma(1)$. Thus $\pi_1(\mathcal{X},\mx{x}_0)$ 
is generated by three elements, which have representatives 
lying in the lines $\check\gamma(\omega)$, 
$\check\gamma(\omega^2)$, and $\check\gamma(1)$. 
Call these, respectively, $\eta_1$, $\eta_2$, and 
$\eta_3$, so that $\eta_k$ wraps around 
$b_k \cap \check\gamma(\omega^k)$. 
(See Figure~\ref{F:pi1}, right.)

\begin{figure}[h]\centering
\includegraphics{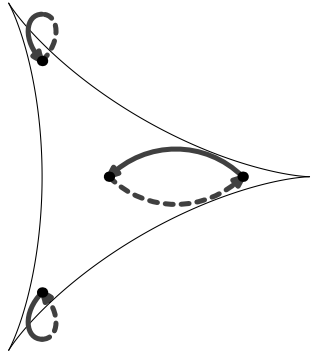}
\caption{The four lifts of $\eta_3$ by $f$. 
The loops lie in $\check\gamma(-1)$, and the arcs 
lie in $\check\gamma(1)$.}\label{F:arclifts}
\label{F:lifts}
\end{figure}

\begin{lemma}\label{L:order2}
For each $k = 1,2,3$ and for all $n \ge 1$, 
$\mu_{f^n}([\eta_k])$ has order $2$.
\end{lemma}
\begin{proof}
We want to show that every lift of $\eta_k$ by 
every iterate of $f$ is either a closed loop, or 
forms a closed loop with one other lift. We will 
use the fact that every lift of $\eta_k$ by any 
iterate of $f$ is contained in some line 
$\check\gamma(t) \in \mathcal{D}^\vee$.

The line $\check\gamma(t)$, when $t \in \C\setminus\{0\}$, 
can be parametrized by 
\[
\sigma_t(s) 
= \left(t + \frac{s}{\sqrt{t}}, \frac{1}{t} + s\sqrt{t}\right)\text, 
\qquad s \in \C\text,
\]
as may be checked directly from the equation for 
$\check\gamma(t)$. (Here, $\sqrt{t}$ can be either 
square root of $t$.) This parametrization of $\gamma(t)$ 
has the nice feature that when $s = 0$, the resulting 
point lies on the critical locus $\mathcal{C}$, since 
it is the midpoint of $\gamma(\sqrt{t})$ and 
$\gamma(-\sqrt{t})$ (see property (B) from 
Section \ref{S:curve}).

Now when we apply $f$ to $\sigma_t(s)$, we obtain 
\[
f\big(\sigma_t(s)\big) 
= \left( 
  \frac{1}{t^2} + (s^2 - 2) t, 
  t^2 + (s^2 - 2)\frac1t
  \right)
= \sigma_{1/t^2}(s^2 - 2)\text.
\]
%Thus $f$ has a kind of skew-product structure over 
%$\mathcal{D}$, with the polynomial $s^2 - 2$ 
%acting on the ``fibers'' (i.e., tangent lines). 
So we just need to consider the possible lifts of 
a closed curve in $\C$ by the polynomial $T(s) = s^2 - 2$, 
avoiding the post-critical set of $T$. (See, for example, 
Figure~\ref{F:arclifts}, which illustrates the four lifts 
of $\eta_3$ by $f$.)

The critical point of $T(s)$ is $0$, and its critical 
value is $-2$. The image of $-2$ by $T(s)$ is $2$, which 
is a fixed point. Let $\eta$ be any loop in $\C$ 
that does not pass through $-2$ or $2$. If $\eta$ does not 
encircle $-2$, then it lifts to a pair of disjoint loops; 
if $\eta$ encircles $2$, then one of these loops encircles 
$2$ and one encircles $-2$, otherwise neither lift encircles 
$-2$. If $\eta$ does encircle $-2$, then it lifts to a double 
cover of itself, consisting of two arcs, that does not encircle 
$-2$.
\end{proof}

In other words, Lemma~\ref{L:order2} says that 
the square of each generator $[\eta_k]$ is in 
$\kappa_f$. Together with the relations in 
$\pi_1(\mathcal{X},\mx{x}_0) = \bar{A}_2$, 
this result implies that $\mathrm{IMG}(f)$ 
is a quotient of $\tilde{A}_2$.
To complete the proof of Theorem~\ref{T:img}, 
we need to show that no additional relations are 
present in $\mathrm{IMG}(f)$.

\begin{proof}[Proof of Theorem~\ref{T:img}]
Recall the realization of $\tilde{A}_2$ as the 
group generated by reflections $\rho_1$, $\rho_2$, 
$\rho_3$ across the sides of an equilateral triangle. 
This group can be expressed as the semidirect product 
$\Lambda \rtimes D_3$, where $\Lambda$ is the 
normal subgroup consisting of translations 
(isomorphic to $\Z^2$) and $D_3$ is the subgroup 
that fixes a vertex of the triangle (the dihedral 
group of order $6$). $D_3$ is generated by the 
reflections in two adjacent sides of the triangle. 

Suppose $\phi : \tilde{A}_2 \to \IMG(f)$ is the 
homomomorphism that sends $\rho_k$ to $[\eta_k]\kappa_f$. 
If $\ker\phi \cap D_3 \ne \{\id\}$, then the order 
of $\phi(D_3)$ is either $1$ or $2$, because the group 
of rotations is the only nontrivial normal subgroup of 
$D_3$; in either case we must have $\phi(\rho_1) = 
\phi(\rho_2) = \phi(\rho_3)$. On the other hand, if 
$\ker\phi \cap \Lambda \ne \{\id\}$, then 
because this intersection is invariant under the 
action of $D_3$, it must contain two linearly independent 
elements $\lambda_1,\lambda_2$; the group 
$\Lambda/(\lambda_1\Z\oplus\lambda_2\Z)$ is then finite 
and so is $\phi(\Lambda)$.

Therefore, in order to show that $\phi$ is an 
isomorphism, it suffices to show that 
$[\eta_1]\kappa_f \ne [\eta_2]\kappa_f$ and that 
$\IMG(f)$ is infinite. The first condition is easily 
checked by observing that $\mu_f([\eta_1])$ and 
$\mu_f([\eta_2])$ are different permutations of 
$f^{-1}(\mx{x}_0)$. The second condition may be seen 
by restricting our attention to an invariant line 
such as $\gamma(1)$; on this line $f$ behaves like 
the single-variable Chebyshev map $s \mapsto s^2 - 2$, 
and the iterated monodromy group of such a map is known 
to have elements of infinite order (see \cite{vN11}). 
\end{proof}

\end{document}